\documentclass[11pt]{amsart}
\usepackage{amsthm,amsmath,amssymb,amscd,graphics,enumerate, stmaryrd,xspace,verbatim, epic, eepic,url}
\usepackage{fullpage}

\usepackage[usenames,dvipsnames]{color}

\usepackage[colorlinks=true]{hyperref}

\usepackage[active]{srcltx}
\usepackage[all]{xypic}
\SelectTips{cm}{}

\newcommand{\ChJarek}[1]{{\color{black} #1}}

\newcommand{\ChDan}[1]{{\color{black} #1}}

{

   \newtheorem{theorem}[subsubsection]{Theorem}
      \newtheorem*{theorem*}{Theorem}
   \newtheorem{proposition}[subsubsection]{Proposition}

   \newtheorem{lemma}[subsubsection]{Lemma}

   \newtheorem{corollary}[subsubsection]{Corollary}
   
   \newtheorem*{conjecture*}{Conjecture}

}
{\theoremstyle{definition}
          \newtheorem*{exercise*}{Exercise}

   \newtheorem*{example*}{Example}
   \newtheorem{definition}[subsubsection]{Definition}
   
   \newtheorem*{definition*}{Definition}
   
   \newtheorem{remark}[subsubsection]{Remark}

}
\newcommand{\RR}{{\mathbb{R}}}

\newcommand{\CC}{{\mathbb{C}}}
\newcommand{\QQ}{{\mathbb{Q}}}
\newcommand{\NN}{{\mathbb{N}}}

\newcommand{\GG}{{\mathbb{G}}}

\newcommand{\ZR}{{\mathbf{ZR}}}

\newcommand{\bmu}{{\boldsymbol{\mu}}}

\newcommand{\cA}{{\mathcal A}}

\newcommand{\cC}{{\mathcal C}}
\renewcommand{\cD}{{\mathcal D}}

\newcommand{\cG}{{\mathcal G}}

\newcommand{\cI}{{\mathcal I}}
\newcommand{\cJ}{{\mathcal J}}

\newcommand{\cO}{{\mathcal O}}

\newcommand{\fm}{{\mathfrak m}}

\def\<{\langle}
\def\>{\rangle}

\newcommand{\Spec}{\operatorname{Spec}}

\newcommand{\cProj}{{{\mathcal P}roj}}

\def\:{{\colon}}
\def\.{{,\dots,}}

\newcommand{\double}{\genfrac..{0pt}1
{\raise -1pt\hbox{$\scriptstyle\longrightarrow$}}{\raise 3pt\hbox
{$\scriptstyle\longrightarrow$}}}

\renewcommand{\setminus}{\smallsetminus}

\def\tototi{\mathbin{\mathop{\otimes}\limits^{\raise-1pt\hbox
{$\scriptscriptstyle {\rm L}$}}}}

\def\indlim{\mathop{\vrule width0pt height7pt depth
4pt\smash{\lim\limits_{\raise 1pt\hbox to 14.5pt
{\rightarrowfill}}}}}
\def\projlim{\mathop{\vrule width0pt height7pt depth
4pt\smash{\lim\limits_{\raise 1pt\hbox to 14.5pt
{\leftarrowfill}}}}}

\newcommand\displaceamount{3pt}

\newcommand{\doubledown}{\ar@<\displaceamount>[d]\ar@<-\displaceamount>[d]}

\newcommand{\doubleup}{\ar@<\displaceamount>[u]\ar@<-\displaceamount>[u]}

\newcommand{\doubleright}{\ar@<\displaceamount>[r]\ar@<-\displaceamount>[r]}

\newcommand{\ord}{{\operatorname{ord}}}
\newcommand{\inv}{{\operatorname{inv}}}
\newcommand{\maxinv}{{\operatorname{maxinv}}}

\begin{document}
\title{Functorial embedded resolution via weighted blowings up}

\author[D. Abramovich]{Dan Abramovich}
\address{Department of Mathematics, Box 1917, Brown University,
Providence, RI, 02912, U.S.A}
\email{abrmovic@math.brown.edu}
\author[M. Temkin]{Michael Temkin}
\address{Einstein Institute of Mathematics\\
               The Hebrew University of Jerusalem\\
                Giv'at Ram, Jerusalem, 91904, Israel}
\email{temkin@math.huji.ac.il}

\author[J. W{\l}odarczyk] {Jaros{\l}aw W{\l}odarczyk}
\address{Department of Mathematics, Purdue University\\
150 N. University Street,\\ West Lafayette, IN 47907-2067}
\email{wlodar@math.purdue.edu}

\thanks{This research is supported by  BSF grant 2014365, ERC Consolidator Grant 770922 - BirNonArchGeom, and NSF grant DMS-1759514.
}

\date{\today}

\begin{abstract}
We provide a simple procedure for resolving, in characteristic 0, singularities of  a variety $X$ embedded in a smooth variety $Y$ by repeatedly blowing up the worst singularities, in the sense of stack-theoretic weighted blowings up. No  history, no exceptional divisors, and no logarithmic structures are necessary to carry this out; the steps are explicit geometric operations requiring no choices; and the resulting algorithm is efficient.

 A similar result was discovered independently by McQuillan \cite{Marzo-McQuillan}.

\end{abstract}
\maketitle
\setcounter{tocdepth}{1}


\section{Introduction} 

\subsection{Statement of result} We consider a smooth variety $Y$ of  dimension $n$ over a field $k$ \emph{of characteristic $0$,} and a reduced closed subscheme $X \subset Y$ of pure codimension $c$; or more generally a closed substack $X$ of pure codimension $c$ of a smooth Deligne--Mumford stack $Y$. Our goal is to resolve singularities of $X$ embedded in $Y$, revisiting Hironaka's  \cite[Main Theorem I]{Hironaka}.

Pairs $X \subset Y$ of possibly different dimensions form a category by considering \emph{surjective} morphisms $(X_1 \subset Y_1) \to (X_2 \subset Y_2)$ of pairs where $f:Y_1 \to Y_2$ is \emph{smooth} and $X_1 = X_2 \times_{Y_2} Y_1$ is the pullback of $X_2$. We in fact define a resolution functor on this category; it is functorial for all smooth morphisms, whether or not surjective, when interpreted appropriately. This follows principles of \cite{Wlodarczyk, Kollar, Bierstone-Milman-funct}.

For a geometric point $p\in |X|$ we defined in \cite[\S 2.12.4]{ATW-principalization}  an upper-semi-continuous function, the \emph{lexicographic order invariant}, which we rescale here and write as:
$$\inv_p(X) = (a_1(p),\ldots, a_k(p))\quad \in \quad \QQ_{\geq 0}^{\leq n}\ \  :=\ \  \bigsqcup_{k\leq n}\QQ_{\geq 0}^{k},$$ ordered lexicographically and taking values in a well-ordered subset.
It detects singularities: the invariant is the sequence $\inv_p(X) = (1,\ldots,1)$ of length $c$ if and only if $p\in X$ is smooth, and otherwise it is bigger. Our invariant  $\inv_p$ is compatible with smooth morphisms of pairs, whether or not surjective: $\inv_p(X_1) = \inv_{f(p)} (X_2)$. The invariant and its properties are recalled in Section \ref{Sec:existence}.

We define $$\maxinv(X) = \max_{p\in |X|} \inv_p(X).$$ This is compatible with \emph{surjective} morphisms of pairs.

In Section \ref{Sec:weighted-blowup} we introduce stack-theoretic weighted blowings up $Y' \to Y$ along centers locally of the form $\bar J=(x_1^{1/ w_1},\ldots, x_k^{1/w_k})$, where$(\ell/w_1,\ldots, \ell/w_k)=\maxinv(X)$ for positive integers  $\ell, w_i$,  and $x_1,\ldots, x_n$ is a carefully chosen regular system of parameters. \ChDan{The centers are supported along smooth loci and their blowings up are smooth.}

The aim of this paper is to prove, \ChDan{using these centers,} the following:

\begin{theorem}
\label{maintheorem} There is a functor $F_{er}$, on pairs with smooth surjective morphisms, associating to a pair $X\subset Y$ \ChDan{as above over a field in characteristic 0, with $X$ \emph{singular},}
a  center $\bar J$ with weighted blowing up $Y'\to Y$ and proper transform $F_{er}(X\subset Y) = (X' \subset Y')$, such that \ChDan{$Y'$ is again a smooth stack  and} $\maxinv(X') <\maxinv(X)$.
In particular there is an integer $n$ so that the iterated application $(X_n \subset Y_n):= F_{er}^{\,\circ n}(X\subset Y)  $ of $F_{er}$ has $X_n$ smooth.

The stabilized functor $F_{er}^{\,\circ \infty}(X\subset Y)$ is functorial for all smooth morphisms of pairs, whether or not surjective.
\end{theorem}

 Using standard arguments, one deduces non-embedded resolution --- see Theorem \ref{Th:non-embedded}. Theorem \ref{maintheorem} relies, \ChDan{through a standard argument,} on principalization of ideals, see Theorem \ref{Th:principalization}.

\ChDan{
\subsection{Our quest for the dream algorithm} We call Theorem \ref{maintheorem} ``the dream algorithm" as we have searched for a result of this nature for some time. In this realized dream, we take a singular variety, with no additional structure;  identify its most singular locus --- in this case the locus of maximal invariant; and perform a simple geometric operation --- a stack theoretic weighted blowing up --- to immediately and visibly improve the singularities.

Note that given any projective resolution $X' \to X$ there is a tautological, one might say nightmarish, ``realization of the dream": one can declare after the fact that the worst singular locus is the support of an ideal $J \subset \cO_X$ whose blowing up is $X' \to X$; such exists by \cite[Theorem 2.7.17]{Hartshorne}.  This answer is unsatisfactory,  as it does not help in finding a resolution. Also the ideal $J$ is complicated and the geometry of its blowing up is, to this date, intractable.

Also, given any resolution algorithm involving simple geometric operations --- such as blowing up smooth centers --- one can tautologically encode the state of the algorithm as an added structure on a variety.  Some earlier algorithms --- such as the algorithms we used in the past --- enriched the data of the ideal $\cI_X$ of the embedded variety $X$ at least with a marking $(\cI_X, a)$ indicating to what extent we intend to improve $\cI_X$; Hironaka's \emph{idealistic  exponents}  \cite{Hironaka-idealistic} are of that nature. This is useful additional data, but insufficient to guide the algorithm or determine the locus to be blown up. Other algorithms considered slightly more  structure: marked ideals with divisors, see \cite{Bierstone-Milman}; \emph{basic objects}, see \cite{Encinas-Villamayor}. The more elaborate structure of \emph{mobiles} of Encinas--Hauser, see \cite{Encinas-Hauser},  was developed specifically to guide resolution, and  completely determines their algorithm. While elegant, it is not a simple structure in the sense we seek,  and only used  for the purpose of resolution. The reason algorithms which use smooth centers require such structures is recalled in Section \ref{revissec} below.

In contrast, in the present paper the locus of maximal singularity is determined in simple terms by the variety alone. Its blowing up is a simple geometric operation   not much different from smooth blowing up.

The result is a  Deligne--Mumford stack, hence our algorithm must allow these as input.   This is natural to us, as Deligne--Mumford stacks are used  in much of today's algebraic geometry. Perhaps the same should be  true in today's research on resolution, as we now discuss.

\subsection{Weighted blowings up, stacks, and resolutions} \label{Sec:weighted-efficient}
Weighted blowings up in a scheme theoretic sense have been used in birational geometry for a long time. Varchenko used them to characterize the {log canonical threshold} of a surface, see \cite{Varchenko}, \cite[Theorem 6.40]{Kollar-et-al}. Kawamata \cite[Appendix]{Shokurov} used them to relate discrepancies to indices.  Reid  \cite{Reid} employs them in  the geometry of surfaces. Mart\'{\i}n-Morales \cite{Martin-Morales-zeta,Martin-Morales-Qres} uses them to efficiently study monodromy zeta functions as well as explicit $\QQ$-desingularizations of certain singularities. Artal-Bartolo, Mart\'{\i}n-Morales, and Ortigas-Galindo \cite{Artal-Martin-Ortigas-V, Artal-Martin-Ortigas-Q}  further study the geometry of surfaces. All this on top of the enormous literature on weighted projective spaces.

All these authors show that weighted blowings up are remarkably efficient in computing invariants of singularities. In \cite{Martin-Morales-zeta,Martin-Morales-Qres}, they are shown,  in a wide class of examples, to be remarkably efficient in finding $\QQ$-resolutions, namely modifications with at most quotient singularities.

Most relevant to the present paper, Panazzolo \cite{Panazzolo} used scheme theoretic  weighted blowings up to simplify foliations in dimension three, and   McQuillan and Panazzolo \cite{McQuillan-Panazzolo} revisited the problem using stack theoretic blowings up.  In particular it is shown there that weighted blowings up are unavoidable for their goals. The paper \cite{McQuillan-Panazzolo} led to the paper \cite{Marzo-McQuillan} concurrent to ours.

In our work, stack theoretic modification appeared in  \cite{ATW-principalization} and shown to be unavoidable for functoriality of logarithmic resolution, leading us to investigate weighted blowings up in general.
}

\subsection{Invariants and parameters} The notation for the present invariant $\inv_p(\cI)$ in \cite{ATW-principalization} was $a_1\cdot \inv_{\cI_X,a_1}(p)$, and extends to arbitrary ideal sheaves on logarithmic orbifolds. Here it is applied solely when $Y$ is smooth with trivial logarithmic structure.

\ChDan{Both this invariant and our center of blowing up are present  in earlier work:}

This invariant \ChDan{$(a_1,\ldots,a_k)$} is closely related to invariants developed in earlier papers on resolution of singularities, in particular Bierstone and Milman's  \cite{Bierstone-Milman} and  W\l odarczyk's \cite{Wlodarczyk}. \ChDan{In fact $(a_1,\ldots,a_k)$ is determined by a sequence $(b_1,\ldots,b_k)$ of integers, which is ``interspersed" in  Bierstone and Milman's richer invariant $(H_1,s_1,b_2,\ldots,b_k,s_k)$. Here $b_1$ is determined by the Hilbert--Samuel function $H_1$ and the $s_i=0$ since no divisors are present ---  our invariant is in essence the classical ``year zero invariant". Invariants of similar nature are already introduced in  \cite[Part II]{Hironaka}.}

The center \ChDan{$J = (x_1^{a_1},\ldots, x_k^{a_k})$} can be interpreted in terms of Newton polyhedra, and as such it appears in Youssin's  \cite[\S 1]{Youssin}, with a closely related precedent in Hironaka's \cite{Hironaka-char}. The local parameters $x_1,\ldots,x_k$ in the definition of $J$ were already introduced in \cite{Bierstone-Milman, Encinas-Villamayor, Wlodarczyk, ATW-principalization} as a sequence of iterated hypersurfaces of maximal contact for appropriate coefficient ideals, see Section \ref{Sec:existence}.  \ChDan{In this paper we prove the necessary properties of the invariant $\inv_p(\cI)$ and the center $J$, but many of these properties are directly implied by these cited works.}



%
%

In earlier work the ideal $(x_1,\ldots,x_k)$ was used to locally define the unique center of blowing up satisfying appropriate admissibility and functoriality properties for resolution using smooth blowings up. A central  observation here is that the stack-theoretic weighted blowing up of $(x_1^{1/w_1},\ldots, x_k^{1/w_k})$ is also functorially associated to $X \subset Y$, see Theorem \ref{Th:center-admissible}(3).

As \ChDan{recalled below}, in general, after blowing up the reduced ideal $(x_1,\ldots,x_k)$, the invariant does not drop, and may increase; Earlier work enhanced this invariant by  including data of exceptional divisors and their history, or more recently, logarithmic structures.
 Another  central  observation here  is that, with the use of weighted blowings up, no history, no exceptional divisors, and no logarithmic structures are necessary.

\subsection{Tools and methods} The present treatment requires the theory of Deligne--Mumford stacks.
An  application of Bergh's Destackification Theorem \cite[Theorem 1.2]{Bergh} or  its  generalization \cite[Theorem B]{Bergh-Rydh}  allows one to replace $X_n \subset Y_n$ by a smooth embedded scheme $X'_n \subset Y'_n$ projective over $X\subset Y$, see Theorem \ref{Th:coarse-resolution}.
\ChJarek{ Alternatively the coarse moduli space admits only abelian singularities and can be resolved directly by
combinatorial methods (see \cite{W-factor,Bergh-Rydh,Wlodarczyk-functorial-toroidal}).} \ChDan{Both destackification and this resolution process apply in arbitrary characteristics, as the stabilizer group-schemes involved are tame.%
\footnote{We remind the reader that, by a theorem of de Jong \cite[Corollary 5.15]{dejong-curves}, as stated in \cite[Theorem 1.4]{Bergh-Rydh}, any variety $X$ over a field of any characteristic admits a purely inseparable alteration $X' \to X$ with $X'$  the coarse moduli space of a smooth Deligne--Mumford stack. Thus, if the field is perfect, resolution of $X$ is reduced to the combination of destackification of a possibly wild Deligne--Mumford stack and the resolution of a purely inseparable cover of a smooth scheme.  }}

\ChDan{Our center $J$ can be identified as an \emph{idealistic exponent}, see \cite{Hironaka-idealistic}, which we present here through} the slightly more flexible  formalism of
 \emph{valuative $\QQ$-ideals}, see Section \ref{Sec:valuative-Q-ideal}, or equivalently ideals in the h topology, see Section \ref{Sec:h-topology}. This formalism allows us to show with little effort that centers are unique and functorial. We believe the formalism, which is inspired by existing work on $\QQ$-ideals,  graded families of ideals, and B-divisors, is the correct formalism to consider ideals with rational multiplicities up to blowings up, a topic permeating birational geometry.

We provide a proof of the theorem based on existing theory of resolution of singularities, using concepts and methods from \cite{Hironaka, Villamayor, Bierstone-Milman, Encinas-Villamayor, Wlodarczyk, Kollar, Encinas-Villamayor-Rees, Bierstone-Milman-funct}, among others. 


\subsection{Concurrent and future work}

The present paper is a beginning for several other works, all requiring additional techniques.

The present treatment does not address logarithmic resolutions, a critical requirement of birational geometry. As Section \ref{Sec:not-dnc} shows this does not follow by accident. The necessary modifications are being worked out in a project under way. This requires bringing in the theory of logarithmic structures as in \cite{ATW-principalization}.

The present results were discovered along the way of our work \cite{ATW-relative}, addressing resolution of singularities in families and semistable reduction, again using the logarithmic theory of \cite{ATW-principalization}. It is our plan to introduce the present methods into that project. Moreover, the entire endeavor must be carried out in the appropriate generality of  \emph{qe schemes} --- this is required in order to deduce results in  other geometric categories of interest, as is done in \cite{Temkin-qe,AT2}.

\ChJarek{

\subsection{Examples: comparing smooth and weighted blowings up}\label{revissec}
\subsubsection{Blowing up without weights} It is well-known that an analogous functorial algorithm with smooth blowings up is impossible, see \cite[Claim 3.6.3]{Kollar}.
We give here slightly different examples.

Consider first the  3-dimensional singularity $$x^2 = y_1y_2y_3.$$ The singular locus
consists of the three lines $x=y_i=y_j=0$, for $i\neq j$, meeting at the origin.
Due to the group of permutations acting on the singularity the only possible \emph{invariant} smooth center is the origin:
  $\{x=y_1=y_2=y_3=0\}$, but its blowing up leads to the three  points with  singularities identical to the original one, occurring  on the three $y_i$-charts:
    Writing $$x = x'y'_3,   y_1 = y'_1y'_3,  y_2 = y'_2y'_3, \text{ and } y_3=y_3'$$ we get, after clearing out $y_3^2$, the equation $${x'}^2 = y'_1y'_2y'_3$$ in the new coordinates.

    Thus  functorial embedded desingularization  by  smooth blowings up,  using no additional structure --- called ``history" by some authors --- is simply impossible, as it may lead to an infinite cycle.\footnote{To resolve this, in Hironaka's classical algorithm one must encode $y'_3=0$ as an exceptional divisor --- this is quite natural and useful. One must then note that upon restriction to the first maximal contact $x=0$ the ideal $y_1'y_1'y_3'$ factors an exceptional ``monomial" part $y'_3$. Unfortunately in general the monomial part makes it impossible to proceed with transverse maximal contact. One must then separate it from the order-2 locus with a resolution  subroutine sometimes called ``the monomial stage". Only then one can find further maximal contact elements and proceed.}

This paucity of functorial centers leads to choices which are far from optimal, and resulting in \emph{worse} singularities.

  Consider the equation  $$x^2 = y_1^ay_2^ay_3^a,$$ with $a\geq 2$ instead. The origin is again the unique possible functorial center, and leads  to a singularity of the form
   $x^2 = y_1^ay_2^ay_3^{3a-2}$ in the $y_3$-chart. This visibly is a worse singularity.





\subsubsection{Weighted blowing up}
 The main reason for working with smooth centers in  Hironaka's approach is that we want to keep the ambient space $Y$ smooth.

 A birational geometer knows that the singularity $x^2=y_1y_2y_3$  asks for the blowing up of  $J=(x^2,y_1^3,y_2^3,y_3^3)$.
 This is the observation used by the authors mentioned in Section \ref{Sec:weighted-efficient} above.  But a weighted blowing up in the \emph{schematic} sense gives rise to a singular ambient space $Y$, with abelian quotient singularities. For the classical algorithm this is a non-starter.

 As explained in Section \ref{Sec:weighted-blowup}, we use instead the \emph{stack theoretic} weighted blowing up of the associated \emph{reduced} center ---  in the example $J^{1/6} = (x^{1/3},y_1^{1/2},y_2^{1/2},y_3^{1/2})$. The chart corresponding to $y_3$ is of the form
 $$[\Spec \CC[x',y_1',y_2',u]/\bmu_2],$$
evidently smooth, where $$y_3=u^2, \ x = x' u^3,\  y_1 = y_1' u^2,\  y_2 = y_2' u^2,$$ and $\bmu_2 = \pm 1 $ acts by $(x',y_1',y_2',u) \mapsto  (-x',y_1',y_2',-u)$. The general equations, and their derivation, are given in Section  \ref{Sec:weighted-blowup}.

 Plugging this into the original equation $x^2=y_1y_2y_3$ we get $u^6{x'}^2 = u^6 y_1'y_2',$ where the factor $u^6$ is exceptional, with proper transform
$${x'}^2 =  y_1'y_2'.$$
In other words, the vector of degrees $(2,3,3,3)$  is reduced to  $(2,2,2)$, an immediate and visible improvement. One more blowing up resolve the singularities.\footnote{Hironaka's classical algorithm requires many more blowings up, and, as indicated in the previous note, is  quite technically  involved}

Similarly, our general algorithm, which requires no knowledge of prior steps taken, assigns to a singularity  a canonical  weighted blowing up which improves actual singularities, rather than a  complex associated structure.  Consequently the natural
centers and resulting valuations are much better suited for computations of various birational invariants, such as log canonical thresholds, as recalled in Section \ref{Sec:weighted-efficient}.

\subsection{Efficiency} In \cite{ATW-principalization} we showed how more limited use of stack-theoretic blowings up leads to a vast improvement in efficiency. In examples the present algorithm is remarkably efficient, with great improvements even on
 \cite{ATW-principalization}. 
 For instance, in the example above, two weighted blowings up suffice. This adds to the evidence recalled in Section \ref{Sec:weighted-efficient}.
Our process is explicitly computable, and an implementation in {\sc Singular} \cite{Singular} is under way.





\subsection{The need for more centers}
Smooth centers with their limitations became an object of  study for possible resolution algorithms in  positive characteristic. The natural  invariants, measured with respect to divisorial coordinates accumulated in the process with smooth centers, behave badly. They lead to numerous phenomena, discovered by Moh \cite{Moh87}, and studied intensely in 
\cite{Abh67,Moh96,Hau96,W08,CP08,CP09,Cut11,BV13,KM16,HP18} 
and many others. However the pathologies  thus discovered  are mostly specific to smooth centers with accumulated divisors and are no longer relevant for the more flexible weighted centers --- especially flexible since we no longer need to keep track of transversality to exceptional loci.

The center we choose here is perfectly attuned to the singularities at hand, at least in characteristic 0: it is the unique center $J$ of maximal invariant which is admissible for the given ideal $\cI$. This description has no mention of characteristic, and one might wonder if, with perhaps an even wider class of centers, one can make inroads into positive characteristic desingularization.

}

\subsection{Acknowledgements} We again mention that  Theorem \ref{maintheorem} was discovered independently by McQuillan \cite{Marzo-McQuillan}. We thank Johannes Nicaise for bringing that to our attention. We thank Michael McQuillan for discussions of past, present and future projects and comparison of approaches (we concluded that the present key results are the same but methods quite different). \ChDan{We also thank Jochen Heinloth, Milena Hering, J\'anos Koll\'ar, Marc Levine, Ronald van Luijk, Ming Hao Quek, David Rydh, and Eugenii Shustin for questions, discussions, and illuminating both theory and existing literature.}

\tableofcontents

\section{Valuative ideals, fractional ideals, and $\QQ$-ideals}
\subsection{Zariski--Riemann spaces} Given an integral noetherian scheme $Y$ we are interested in understanding ideals, and more generally $\QQ$-ideals, as they behave after arbitrary blowing up. For instance the ideals $(x^2,y^2)$ and $(x^2,xy,y^2)$ coincide after blowing up the origin, and a formalism in which they are the same object is desirable.  We propose to work with the Zariski--Riemann space $\ZR(Y)$ of $Y$, the projective limit of all projective birational transformations of $Y$, whose points consist of all valuation rings $R$ of $K(Y)$ extending to a morphism $\Spec R \to Y$.

The space $\ZR(Y)$ carries a constant sheaf $K$, a subsheaf of rings $\cO$ with stalk at $v$ consisting of the valuation ring $R_v$, and a sheaf of ordered groups $\Gamma= K^*/\cO^*$ such that $v: K^* \to \Gamma$ is the valuation. The image $v(\cO\smallsetminus \{0\} ) =: \Gamma_+\subset \Gamma$ is the valuation monoid consisting of non-negative sections of $\Gamma$.

\ChDan{The space $\ZR(Y)$ is quasicompact, see \cite[Proposition 3.2.1]{Temkin-curves}.

If $Y= \cup Y_i$ is reduced but possibly reducible with irreducible components $Y_i$, we define $\ZR(Y):= \sqcup \ZR(Y_i)$.}


\begin{remark} While Theorem \ref{maintheorem} is applied to DM stacks $X \subset Y$, functoriality means that we can always work on an \'etale cover by a scheme $\tilde X \subset \tilde Y$ : the resolution step $F_{er}(X\subset Y)$ is obtained by \'etale descent from $F_{er}(\tilde X\subset \tilde Y)$. In particular we need not introduce $\ZR(Y)$ for a stack. Nevertheless we note that such $\ZR(Y)$ can be constructed as well, be it by \'etale descent, or directly as a limit, or as a suitably normalized fibered product of $Y$ with the Zariski--Riemann space of the coarse moduli space.
\end{remark}

\subsection{Valuative $\QQ$-ideals}\label{Sec:valuative-Q-ideal} By a \emph{valuative ideal} on $Y$ we mean a section $\gamma\in H^0(\ZR(Y), \Gamma_+)$. Every ideal $\cI$ on every birational model $Y' \to Y$, proper over $Y$, defines a valuative ideal we denote $v(\cI)$ by taking the minimal element of the image of $\cI$ in $\Gamma_+$.  Ideals with the same integral closure have the same valuative ideal.
 Every valuative ideal  $\gamma$ defines an ideal sheaf on every such $Y'$ by taking $\cI_\gamma:=\{f\in \cO_{Y'} | v(f) \geq \gamma_v \forall v\}$, which is automatically integrally closed.

By a \emph{valuative fractional ideal} we mean a section $\gamma\in H^0(\ZR(Y), \Gamma)$, not necessarily positive, with similar correspondences. These do not figure in this paper.

The group $\Gamma_\QQ = \Gamma\otimes \QQ$ is also ordered. We denote the monoid of non-negative elements $\Gamma_{\QQ+}$. By a \emph{valuative $\QQ$-ideal} we mean  a section $\gamma\in H^0(\ZR(Y), \Gamma_{\QQ+})$. The definition of  $\cI_\gamma$ extends to this case.
It is a convenient way to consider $\QQ$-ideals, extending the definition of $v(\cI)$: given a finite collection  $f_i\in\cO_Y$ and $a_i \in \QQ_{>0}$ we write
\begin{equation}(f_1^{a_1},\ldots,f_k^{a_k}) \ := \ (\min\{a_i\cdot v(f_i)\})_v\quad \in\quad H^0(\ZR(Y), \Gamma_{\QQ+})\label{Eq:Q-ideal}\end{equation}
for the naturally associated valuative $\QQ$-ideal. When $a_i$ are integers this coincides with $v(f_1^{a_1},\ldots,f_k^{a_k})$.  There is again a similar notion of a valuative fractional $\QQ$-ideal.

As was pointed out by D. Rydh, valuative $\QQ$-ideals are equivalent to effective $\QQ$-Cartier divisors on $\ZR(X)$. Indeed, any section $\gamma$ of $\Gamma_+$ is locally the image of an element of $\cO$, and since $\ZR(X)$ is quasi-compact, finitely many such representatives suffice. \ChDan{Moreover, taking a common birational model $Y'\to Y$ over which all the representative sections are regular, we find that $\gamma$ is an invertible ideal on $Y'$. Allowing denominators, any valuative $\QQ$-ideal $\gamma$ is written,  using the notation of  \eqref{Eq:Q-ideal}, locally on the model $Y'$ as $\gamma = (f^a)$.

\subsection{Idealistic exponents} A valuative $\QQ$-ideal which is represented locally on $Y$ itself as $(f_1^{a_1},\ldots,f_k^{a_k})$ is an \emph{idealistic exponent}. This notion coincides with Hironaka's \cite[Definition 3]{Hironaka-idealistic}  by \cite[Remark (2.2)]{Hironaka-idealistic}. Hironaka's notation $(\cJ, b)$, with $\cJ \subset \cO_Y, b\in \NN$ translates to the valuative $\QQ$-ideal $\cJ^{1/b}$. Hironaka's definition of pullback of an idealistic exponent under a dominant morphism $Y' \to Y$ extends to an arbitrary valuative $\QQ$-ideal. 

As indicated in the next section, these are related to Rees algebras \cite{Encinas-Villamayor-Rees} or graded families of ideals \cite[Section 2.4.B]{LazarsfeldI}. This relationship will be pursued in greater depth elsewhere.
}

\subsection{Centers and admissibility}\label{Sec:centers-admissibility} By a \emph{center} $J$ on $Y$ we mean a valuative $\QQ$-ideal for which there is an affine covering $Y = \cup U_i$ and regular systems  \ChDan{of parameters} $(x^{(i)}_1,\dots,x^{(i)}_k) = (x_1,\dots,x_k)$ on $U_i$ such that   $J_{U_i} = (x_1^{a_1},\ldots,x_k^{a_k})$ for some $a_j \in \QQ_{>0}$ independent of $i$.
\ChDan{In such coordinates it corresponds to a unique monomial valuation associated to the cocharacter $$\left( a_1^{-1},\ldots,  a_k^{-1},0,\ldots,0\right),$$ where  $v(\prod x_i^{c_i}) = \sum_{i=1}^k c_i/a_i. $}

A center $J$ is \emph{admissible} for a valuative $\QQ$-ideal $\beta$ if $J_v\leq \beta_v$ for all $v$. A center is \emph{admissible} for an ideal $\cI$ if it is admissible for the associated valuative $\QQ$-ideal $v(\cI)$.

The center $J$ is \emph{reduced} if $w_i=1/a_i$ are positive integers with $\gcd(w_1,\ldots,w_k)=1$. For any center $J$ we write $\bar J = (x_1^{1/w_1},\ldots,x_k^{1/w_k})$ for the unique reduced center such that $\bar J^\ell = J$ for some $\ell\in \QQ_{>0}$.

In Section \ref{Sec:weighted-blowup} below we define the blowing up of $(x_1^{1/w_1},\ldots,x_k^{1/w_k})$. In Section \ref{Sec:admissibility} we show how admissibility is manifested in terms of this blowing up, and becomes very much analogous to the notion used in earlier resolution algorithms.

\subsection{Relation with the h topology}\label{Sec:h-topology} Valuative $\QQ$-ideals are closely related to ideals in the h topology, where Zariski open coverings and alterations generate a  cofinal collection of  coverings, see \cite[Definition 3.1.5 and Theorem 3.1.9]{Voevodsky}. Indeed if $\cJ$ is an ideal in the h topology represented by an ideal on an h covering  $\{U_i\}$, then the valuative ideals $v(\cJ\cO_{U_i})$ agree on overlaps. Since  for a valuation $v$ on $Y$ and  any valuation $w$ on some $U_i$ over $v$ we have $\Gamma_{\QQ,w} = \Gamma_{\QQ,v}$, this defines a valuative $\QQ$-ideal on $Y$.

\section{Weighted blowings up}\label{Sec:weighted-blowup}

\ChDan{
Stack theoretic projective spectra were considered informally by Miles Reid, introduced officially in  \cite{Abramovich-Hassett} to study moduli spaces of varieties,
and treated in Olsson's book \cite[Section 10.2.7]{Olsson-book}.

Rydh's forthcoming manuscript \cite{Rydh-proj} provides  foundations for stack-theoretic blowings up. The presentation here is rather terse as complete details will appear there.
 The local equations we present here can be found in \cite[Page 167]{Kollar-et-al}, where they are developed for  the study of log canonical thresholds. The graded algebras we present below are special cases of the graded families of ideals discussed in \cite[Section 2.4.B]{LazarsfeldI}, especially Example 2.4.8.
}
\subsection{Graded algebras and their $\cProj$}
Given a quasicoherent graded algebra $\cA = \oplus_{m\geq 0} \cA_m$ on $Y$ with associated $\GG_m$-action defined by $(t, s) \mapsto t^m s$ for  $s \in \cA_m$   we define its stack-theoretic projective spectrum to be $$\cProj_Y\cA := [(\Spec_{\cO_Y} \cA \setminus S_0) /\GG_m],$$ where the vertex $S_0$ is the zero scheme of the ideal $\oplus_{m>0} \cA_m$. When $\cA_1$ is coherent and generates $\cA$ over $\cA_0$ this agrees with the Construction in \cite[II.7, page 160]{Hartshorne}. As usual $\cProj_Y\cA$ carries an invertible sheaf $\cO_{\cProj_Y\cA}(1)$ corresponding to the graded module $\cA(1)$. When $\cA$ is finitely generated over $\cO_Y$ with coherent graded components  the resulting morphism $\cProj_Y\cA \to Y$ is proper.

\subsection{Rees algebras of ideals} If $\cI$ is an ideal, its Rees algebra is $\cA_\cI := \oplus_{m\geq 0} \cI^m$, and the blowing up of $\cI$ is $Y'=Bl_Y(\cI):=\cProj_Y(\cA_\cI)$. It is the universal birational map making $\cI\cO_{Y'}$ invertible, in this case $Y'\to Y$ projective, see Definition \cite[II.7, page 163]{Hartshorne}.

\subsection{Rees algebras of valuative $\QQ$-ideals}
Now let $\gamma$ be a valuative $\QQ$-ideal, and 
define its Rees algebra to be $$\cA_\gamma := \bigoplus_{m\in \NN } \cI_{m\gamma}.$$

The blowing up of $\gamma$ is defined to be $Y' =Bl_Y(\gamma):= \cProj_Y\cA_\gamma$.

\ChDan{At least when $\gamma = (f_1^{a_1},\ldots, f_k^{a_k})$ is an idealistic exponent,}  $Y'\to Y$  satisfies a corresponding universal property. Since we will not use this property in this paper, we just mention that the valuative $\QQ$-ideal $E = \gamma\cO_{Y'}$, in a suitable sense of Zariski--Riemann spaces of stacks, or as an h-ideal, becomes  an \emph{invertible ideal sheaf} on $Y'$. We only show this below for the blowing up of a center.


Note that if $Y_1 \to Y$ is flat and $Y_1' = Bl_Y(\gamma \cO_{Y_1})$ then $Y_1' =Y' \times_Y Y_1$.

\subsection{Weighted blowings up: local equations}\label{Sec:weighted-equations}

Now consider the situation where $\gamma$ is a center of the special form $J= (x_1^{1/w_1},\ldots,x_k^{1/ w_k})$, with $w_i\in \NN$.
In this case the algebra $\cA_\gamma = \bigoplus_{m\in \NN } \cI_{m\gamma}$, with $\cI_{m\gamma} = (x_i^{b_1}\cdots x_n^{b_n} | \sum w_ib_i\geq m)$ \ChDan{is finitely generated. It} is the integral closure inside $\cO_Y[T,T^{-1}]$ of the simpler algebra with generators $(x_i)T^{w_i}$. We can therefore describe $Bl_Y(J)= Bl_Y(\gamma)$, which deserves to be called a stack-theoretic weighted blowing up, explicitly in local coordinates, as follows:

The chart associated to $x_1$ has local variables $u, x_2', \ldots, x_n'$, where
\begin{itemize} \item $x_1= u^{w_1}$,
\item $x_i' = x_i / u^{w_i}$ for $2 \leq i \leq k$, and
\item $x_j' = x_j$ for $j>k$.
\end{itemize}
The group $\bmu_{w_1}$ acts through $$(u, x_2',\ldots,x_k')\quad  \mapsto\quad  (\zeta_{w_1}u,\ \zeta_{w_1}^{-w_2} x_2',\ \ldots,\ \zeta_{w_1}^{-w_k} x_k')$$ and trivially on $x_j', j>k$, giving an \'etale local isomorphism of the chart with $$[\Spec k[u, x_2', \ldots, x_n'] / \bmu_{w_1}].$$ It is easy to see that these charts glue to a stack-theoretic modification $Y'\to Y$ with a smooth $Y'$ and its coarse space is the classical (singular) weighted blowing up.

Write $E=(u)$ for the exceptional ideal. Then $v(E) = (x_1^{1/w_1},\ldots,x_k^{1/w_k})$, and this persists on all charts, in other words the center $(x_1^{1/w_1},\ldots,x_k^{1/w_k})$ becomes an invertible ideal sheaf on $Y'$.

We sometimes, but not always, insist on $\gcd(w_1,\ldots,w_k) = 1$, in which case the center is \emph{reduced}. We will however need to consider the proper transform of the locus $H = \{x_1=0\}$, where it may happen that $\gcd(w_2,\ldots,w_k) \neq 1$.  The relationships are summarized by the following lemma, which follows by considering the charts:

\begin{lemma}\label{Lem:root}  If $J'= (x_1^{1/w_1},\ldots,x_k^{1/ w_k})$ and $J''= (x_1^{1/cw_1},\ldots,x_k^{1/c w_k})$ with $w_i,c$ positive integers, and if  $Y', Y'' \to Y$ are the corresponding blowings up, with $E',E''$ the exceptional divisors, then $Y'' = Y'(\sqrt[c]{E'})$ is the root stack of $Y'$ along $E'$.

Write $H = \{x_1=0\}$, and $H'\to H$ the blowing up of the \emph{reduced} center $\bar J'_H$ associated to $J'_H := (x_2^{1/w_2}, \ldots, x_k^{1/w_k})$, with exceptional $E_H$. Then the proper transform $\tilde H' \to H$ of $H$ via the blowing up of $J''$ is the root stack $H'(\sqrt[(cc')]{E_H})$ of $H'$ along $E_H\subset H'$, where
 \ChDan{$c'= \gcd(w_2,\ldots,w_k)$.}
  Therefore $\tilde H'$ is the blowing up of $\bar {J'_H}^{1/(cc')}$.
\end{lemma}


\subsection{Derivation of equations} Let us derive the description in Section \ref{Sec:weighted-equations} above. Write $y_i = x_iT^{w_i}$. The $x_1$-chart is the stack $[\Spec \cA[y_1^{-1}] / \GG_m]$. The slice $W_1:= \Spec \cA[y_1^{-1}]  / (y_1-1)$ is stabilized by $\bmu_{w_1}$, so the embedding $W_1 \subset \Spec \cA[y_1^{-1}]$ gives rise to a morphism $\phi:[W_1 /\bmu_{w_1}] \to  [\Spec \cA[y_1^{-1}] / \GG_m]$. This is an isomorphism: the equation $u^{w_1} = x_1$ describes a $\bmu_{w_1}$-torsor on $\Spec \cA[y_1^{-1}]$ mapping to $W_1$ equivariantly via $T \mapsto u^{-1}$. The resulting morphism $\Spec \cA[y_1^{-1}] \to [W_1 /\bmu_{w_1}]$ descends to $[\Spec \cA[y_1^{-1}] / \GG_m] \to  [W_1 /\bmu_{w_1}]$ which is an inverse to $\phi$.

It thus remains to show that  $[W_1 /\bmu_{w_1}]$ has the local description above. Since $T^{-w_1} = y_1^{-1} x_1\in \cA[y_1^{-1}]$ and $\cA$ is integrally closed in  $ \cO_Y[T,T^{-1}]$ we have $u := T^{-1} \in \cA[y_1^{-1}]$, and its restriction to $W_1$ satisfies $u^{w_1} = x_1$. For $i=2,\ldots,k$ we write $x_i'$ for the restriction of $y_i$, obtaining $x_i' = x_i / u^{w_i}$. Now $W_1$ is normal and finite birational over $\Spec k[u, x_2',\ldots,x_n']$, hence they are isomorphic.

\subsection{Weighted blowings up: local toric description}

Again working locally, assume that $Y = \Spec k[x_1,\ldots,x_n]$. It is the affine toric variety associated to the monoid $\NN^n \subset \sigma = \RR_{\geq 0}^n$. Here the generator $e_i$ of $\NN^n$ corresponds to the monomial valuation $v_i$ associated to the divisor $x_i=0$, namely $v_i(x_j) = \delta_{ij}$.

The monomial $x_i^{1/w_i}$ defines the linear function on $\sigma$ whose value on $(b_1,\ldots,b_n)$ is its valuation $b_i/w_i$. The ideal $(x_1^{1/w_1},\ldots,x_k^{1/w_k})$ thus defines the piecewise linear function $\min_i\{b_i/w_i\}$, which becomes linear precisely on the star subdivision $\Sigma=v_{\bar J}\star\sigma$ with
$$v_{\bar J} =  (w_1,\ldots,w_k,0,\ldots,0).$$ This defines the scheme theoretic weighted blowing up $\bar Y'$. Note that this cocharacter $v_{\bar J}$ is a multiple of the valuation associated to the exceptional divisor of the center.

Since $v_{\bar J}$ is assumed integral, we can apply the theory of toric stacks \cite{BCS, FMN, GSI,GSII,Gillam-Molcho-stacky}. We have a smooth toric stack $Y' \to \bar Y'$ associated  to the same fan $\Sigma$ with the cone $\sigma_i= \langle v_{\bar J},e_1,\ldots, \hat e_i,\ldots,e_n\rangle$ endowed with the sublattice $N_i \subset N$ generated by the elements $ v_{\bar J},e_1,\ldots,\hat e_i,\ldots,e_n$, for all $i=1,\ldots,k$. This toric stack is precisely the stack theoretic weighted blowing up $Y' \to Y$. One can derive the equations in Section \ref{Sec:weighted-equations} from this toric picture.

\section{Coefficient ideals}\label{Sec:coefficient}
\subsection{Graded algebra and coefficient ideals}

Fix an ideal $\cI$ and an integer $a>0$.
We use the notation of  \cite{ATW-principalization}, except that we use the saturated coefficient ideal as in \cite{Kollar, ATW-relative}, which is consistent with the Rees algebra approach of \cite{Encinas-Villamayor-Rees}: Consider the graded subalgebra \ChDan{$\cG= \cG(\cI,a) \subseteq \cO_Y[T]$} generated by placing $\cD^{\leq a-i}\cI$ in degree $i$.  \ChDan{Its graded pieces are
$$\cG_j = \sum_{\sum_{i=0}^{a-1} (a-i)\cdot b_i \ \geq j} \cI^{b_0}\cdot (\cD^{\leq 1} \cI)^{b_1} \cdots (\cD^{\leq a-1} \cI)^{b_{a-1}},$$
where the sum runs over all monomials in the ideals $\cI ,\ldots, \cD^{\leq a-1} \cI$}
 of weighted degree $$\sum_{i=0}^{a-1} (a-i)\cdot b_i \quad \geq \quad j.$$

 \ChDan{ The product rule, and the trivial inclusion $\cD^{\leq 1}\cD^{\leq a-1}\cI \subset (1)$, imply that $\cD\cG_{k+1} \subset \cG_k$ for $k\geq 0$.
  }


\begin{definition} Let $\cI \subset \cO_Y$ \ChDan{and $a\geq 1$ an integer}. Define the \emph{coefficient ideal} $$C(\cI,a)  := \cG_{a!}.$$
\end{definition}

The formation of $\cG$ and $C(\cI,a)$ is functorial for smooth morphisms: if $Y_1 \to Y$ is smooth then $C(\cI,a)\cO_{Y_1}=C(\cI\cO_{Y_1},a)$. This follows since the formation of $\cD^{\leq 1} \cI$, ideal product, and ideal sum are all functorial.


\ChDan{For the rest of the section we assume that $\cI \subset \cO_Y$ has maximal order $\leq a$.}

\subsection{Maximal contact} Recall that an element  $x\in \cD^{\leq a-1}\cI$ \ChDan{which is a regular parameter} at $p\in Y$ is called  \emph{a maximal contact element} at $p$, and its vanishing locus \emph{a maximal contact hypersurface}. The coefficient ideal combines sufficient information from derivatives of $\cI$ so that when one restricts $C(\cI,a)$ to a hypersurface of maximal contact no necessary information is lost.

For completeness, any parameter is a maximal contact element for the unit ideal.

\subsection{Invariance}
Now consider $\cI\subset \cO_Y$ and assume $x_1 \in \cD^{\leq a-1} \cI$ is a maximal contact element at $p\in Y$. The ideals $\cG_i$ are all MC-invariant in the sense of \cite[\S 3.53]{Kollar}: $\cG_1\cdot \cD^{\leq 1}\cG_i \subset \cG_i$, hence they are \emph{homogeneous} in the sense of \cite{Wlodarczyk}:

\begin{theorem}\label{Th:homogeneous} Let $x_1,x_1'$ be maximal contact elements at $p$, and $x_2,\ldots,x_n\in \cO_{Y,p}$ such that $(x_1,x_2,\ldots,x_n)$ and  $(x_1',x_2,\ldots,x_n)$ are both regular sequences \ChDan{of parameters}. There is a scheme $\tilde Y$ with point $\tilde p\in \tilde Y$  and two morphisms $\phi,\phi': \tilde Y \to Y$  with $\phi(\tilde p) = \phi'(\tilde p) = p$, both \'etale at $p$, satisfying
\begin{enumerate}
\item $\phi^*x_1 = {\phi'}^* x_1'$,
\item $\phi^*x_i = {\phi'}^* x_i$ for $i=2,\ldots,n$, and
\item $\phi^*\cG_i = {\phi'}^* \cG_i$.
\end{enumerate}
\end{theorem}
This is \cite[Theorem 3.92]{Kollar}, generalizing \cite[Lemma 3.5.5]{Wlodarczyk}.\footnote{These are the easier properties of coefficient ideals. We emphasize that we do not require the harder part (4) of \cite[Lemma 3.5.5]{Wlodarczyk} or \cite[Theorem 3.97]{Kollar} describing the behavior after a sequence of blowings up.}


\subsection{Formal decomposition}
We now pass to formal completions. Extending to a regular sequence \ChDan{of parameters} we write $\hat\cO_{Y,p}=k\llbracket x_1,\ldots,x_n\rrbracket$. We use the reduction homomorphism $k\llbracket x_1,\ldots,x_n\rrbracket \to k\llbracket x_2,\ldots,x_n\rrbracket$ and the inclusion $k\llbracket x_2,\ldots,x_n\rrbracket \to k\llbracket x_1,\ldots,x_n\rrbracket$.

We have $\cG_j = (x_1^j) + (x_1^{j-1}) \cG_1 + \cdots + (x_1) \cG_{j-1} + \cG_j$ since the ideal on the left contains every term on the right. Write $\bar\cC_j = \cG_{j} k\llbracket x_2,\ldots,x_n\rrbracket$ via the reduction homomorphism and $\tilde\cC_j = \bar\cC_j k\llbracket x_1,\ldots,x_n\rrbracket$ via inclusion.


\begin{proposition}
 \label{Lem:coefficient-description} After passing to completions, \ChDan{we have $$\cG_j = (x_1^j) + (x_1^{j-1}) \tilde\cC_1 + \cdots + (x_1) \tilde\cC_{j-1} + \tilde\cC_j,$$} in particular
 $$C(\cI,a) = (x_1^{a!}) + (x_1^{a!-1} \tilde\cC_{1})  + \cdots + (x_1 \tilde\cC_{a!-1}) + \tilde\cC_{a!}.$$
\end{proposition}
\ChDan{
\begin{proof}
We write $x= x_1$. Apply induction on $j$, noting that $\cG_0 = (1)$ so that we may start with $(1) = \tilde\cC_0$ and inductively assume the equality holds up to $j-1$.

For an integer $M>j$ the ideals $\cG_j\supset (x^M)$ are stable under  the linear operator $x\partial/\partial x$. Hence the quotient $\cG_j / (x^M)$ inherits a linear action, with $m$-eigenspaces we denote  $x^m \cdot \cG_j^{(m)} \subset x^m k\llbracket x_2,\ldots,x_n\rrbracket$, giving

$$ \cG_j / (x^{M+1}) \quad = \quad \cG_j^{(0)}\  \oplus \ x \cdot \cG_j^{(1)}\  \oplus \dots  \oplus\  x^{m} \cdot \cG_j^{(m)}\  \oplus \dots  \oplus \ x^{M} \cdot \cG_j^{(M)},$$ with $ \cG_j^{(m)} \subset k\llbracket x_2,\ldots,x_n\rrbracket$ and equality holding for $m\geq j$. Note that $\cG_j^{(0)} = \bar \cC_j$.

  The subspaces  $\cG_j^{(m)} \subset k\llbracket x_2,\ldots,x_n\rrbracket$ are independent of the choice of $M\geq m$. Moreover $x^j \cdot \cG_j^{(m)} \subset \cG_j \cap x^j \cdot k\llbracket x_2,\ldots,x_n\rrbracket$, so that $$\cG_j^{(m)} \ = \ \frac{\partial^j}{\partial x^j} (x^j \cdot \cG_j^{(m)})\  \subset \ \cG_{j-m} \cap k\llbracket x_2,\ldots,x_n\rrbracket\  \subset \ \bar \cC_{j-m}.$$
  Taking ideals we obtain
  $$\cG_j \quad\subset \quad \cG_j^{(0)}\  + \ (x) \tilde\cC_{j-1}\ + \dots  +\  (x^{j-1})  \tilde\cC_1\ + (x^j) .$$
  Induction gives $$(x) \tilde\cC_{j-1}\ + \dots  +\  (x^{j-1})  \tilde\cC_1\ + (x^j) \quad  =\quad  (x) \cG_{j-1} \quad \subset\quad  \cG_j.$$ Together with $\bar\cC_j = \cG_j^{(0)} \subset \cG_j$ the equality follows.
\end{proof}

%
%
%
}

\section{Invariants, centers, and admissibility}

\subsection{Existence of invariants and centers}\label{Sec:existence}
Fix an ideal $\cI=\cI[1]$ and $p\in Y$. We define a finite sequence of integers $b_i$, rational numbers $a_i$, and parameters $x_i$.

  If $\cI_p=(0)$ set $\inv_p(\cI)=()$ to be the empty sequence, with an empty sequence of parameters.

 Otherwise set $a_1=b_1:=\ord_p(\cI[1])$, and take the parameter $x_1$ to be a maximal contact element at $p$. Inductively one writes $\cI[{i+1}] = C(\cI[{i}], b_i)|_{V(x_1,\ldots,x_i)}$, the restricted coefficient ideal, with order $\ord_p(\cI[{i+1}]) = b_{i+1}$, one sets $a_{i+1} = b_{i+1}/(b_i-1)!$, and one takes $x_{i+1}$ a lifting to $Y$ of the maximal contact element for $\cI[{i+1}]$.

 Equivalently, $\inv_p(\cI[1]) = \left(a_1, \inv_p(\cI[2]) / (a_1-1)!\right)$ the concatenation, and $x_2,\ldots$ are lifts of the parameters for $\cI[2]$. In the notation of the previous section $\cI[2] = \bar \cC_{a_1!}$.

 Note in particular that if $\cI[2]=0$ then $\inv_p(\cI) = (a_1)$ with parameter $x_1$.
 
\ChDan{We remark that once Theorem \ref{Th:center-admissible}  is proven, we can use a better definition: this is the maximal invariant of a center admissible for $\cI$.}

\ChDan{Invariants are ordered lexicographically, with truncated sequences considered larger, for instance $$(1,1,1)<(1,1,2)<(1,2,1)<(1,2)<(2,2,1).$$ }The invariant takes values in the well-ordered subset $\Gamma_n$, \ChDan{since it is order-equivalent to $(b_1,\ldots,b_k)$. Explicitly write}  $\Gamma_1 = \NN^{\geq 1}$ and $$\Gamma_n\quad  = \quad \Gamma_1 \ \ \sqcup\ \ \bigsqcup_{a\geq 1} \{a\} \times \frac{\Gamma_{n-1}}{(a-1)!}.$$



 \begin{theorem}[\cite{ATW-principalization}]\label{Th:invariant} The invariant $\inv_p$ is independent of the choices. It is upper-semi-continuous. It is functorial for smooth morphisms: if $f:Y_1 \to Y$ is smooth and $p' \in Y'$ then $\inv_{p'}(\cI\cO_{Y_1}) = \inv_{f(p')}(\cI)$.
\end{theorem}

 \begin{proof}

Since both $\ord_p(\cI)$ and the formation of coefficient ideals are functorial for smooth morphisms, the invariant is functorial for smooth morphisms, once parameters are chosen. We now show that the parameter choices do not change the invariant.

 The integer $a_1 = \ord_p(\cI) = \max\{a:  \cI_p\subseteq\fm_p^a \}$ requires no choices. Given a regular sequence \ChDan{of parameters} $(x_1,\ldots,x_n)$ extending $(x_1,\ldots,x_k)$, and given another maximal contact element $x_1'$, we may choose constants $t_i$, and  replace $x_2,\ldots, x_n$ by $x_2+t_2 x_1,\ldots, x_n+t_n x_1$ so that also $(x_1',x_2,\ldots,x_n)$ is a regular sequence \ChDan{of parameters}.

 Taking \'etale $\phi,\phi':\tilde Y \to Y$ as in Theorem \ref{Th:homogeneous},  we have $\phi^*\cI[2] = {\phi'}^*\cI[2]'$, where $\cI[2]'$ is defined using $x_1'$.
 By induction $a_2,\ldots,a_k$ are independent of choices. Hence $(a_1,\dots,a_k)$ is independent of choices.

Since the closed subscheme $V(\cD^{\leq a-1} \cI)$ is the locus where $\ord_p(\cI)\geq a$, the order is upper-semi-continuous. The subscheme $V(\cD^{\leq a_1-1} \cI)$  is contained in    $V(x_1)$ on which $\inv_p(\cI[2])$ is upper-semi-continuous by induction, hence $\inv_p(\cI)$ is upper-semi-continuous.

\end{proof}

We say that the center $J=(x_1^{a_1},\ldots,x_k^{a_k})$ formed by the invariant $(a_1,\dots,a_k)$ and the chosen parameters $(x_1,\ldots,x_k)$ is \emph{associated to $\cI$ at $p$}. This notion is functorial for smooth morphisms, once parameters are chosen on $Y$. We will show in Theorem \ref{Th:center-admissible}(3) that the center is uniquely determined as a valuative $\QQ$-ideal: \ChDan{it is the unique center of maximal invariant admissible for $\cI$}. For the time being we note the following  consequence of Theorem \ref{Th:homogeneous}:

\begin{corollary}\label{Lem:invariant-transverse} If $x_1'$ is another maximal contact element such that $(x_1',x_2,\ldots,x_n)$ is a regular sequence \ChDan{of parameters}, then  $J'=({x_1'}^{a_1},x_2^{a_2}\ldots,x_k^{a_k})$ is also a center associated to $\cI$ at $p$. 
\end{corollary}
This  again follows since  $\phi^*\cI[2] = {\phi'}^*\cI[2]'$, where $\cI[2]'$ is defined using $x_1'$.

\subsection{Admissibility of centers}\label{Sec:admissibility} As in earlier work on resolution of singularities, \emph{admissibility} allows flexibility in studying the behavior of ideals under blowings up of centers. This becomes important when an ideal is related to the sum of ideals with different invariants of their own, but all admitting a common admissible center.

In this section we assume that  $a_1$ is a positive integer and $a_i\leq a_{i+1}$. We deliberately do not assume $(a_1,\ldots,a_k)$ is $\inv_p(\cI)$ --- see Remark \ref{Rem:flexible}.

\subsubsection{Admissibility and blowing up}\label{Sec:admissibility-blowup} As in Section \ref{Sec:centers-admissibility} we say that a center $J=(x_1^{a_1},\ldots,x_k^{a_k})$ is $\cI$-admissible at $p$ if the inequality $(x_1^{a_1},\ldots,x_k^{a_k})\leq v(\cI)$ of valuative $\QQ$-ideals is satisfied on a neighborhood of $p$.

Very much in analogy to the notion used in earlier resolution algorithms, this can be described in terms of the weighted blowing up $Y' \to Y$ of the reduced center   $\bar J:= (x_1^{1/w_1},\ldots,x_k^{1/w_k})$, with $w_i$ integers with $\gcd(w_1,\ldots,w_k) = 1$ as follows: let $E = \bar J \cO_{Y'}$, which is an invertible ideal sheaf. Note that since $a_1 w_1$ is an integer also $J\cO_{Y'} = E^{a_1w_1}$ is an invertible ideal sheaf. Therefore $J=(x_1^{a_1},\ldots,x_k^{a_k})$ is $\cI$-admissible if and only if $E^{a_1w_1}$ is $\cI \cO_{Y'}$ admissible, if and only if $\cI \cO_{Y'} = E^{a_1w_1} \cI'$, with $\cI'$ an ideal.

\ChDan{
When $J$ is the center associated to $\cI$, which is shown to be admissible below, the ideal $\cI'$ is called the \emph{weak transform of $\cI$}.}

In terms of its monomial valuation, $J$ is admissible for $\cI$ if and only if $v_J(f) \ge 1$ for all $f \in \cI$. This means that if $f = \sum c_{\bar\alpha} x_1^{\alpha_1} \cdots x_n^{\alpha_n}$ then $\sum_{i=1}^k \alpha_i/a_i\geq 1$ whenever $ c_{\bar\alpha}\neq 0$. \ChDan{This is convenient for testing admissibility, as long as one remembers that $v_{J^m} = v_J / m$.}

If $Y_1 \to Y$ is smooth and  $J$ is  $\cI$-admissible then $J \cO_{Y_1}$ is $\cI\cO_{Y_1}$-admissible, with the converse holding when $Y_1 \to Y$ is surjective.

\subsubsection{Working with rescaled centers}
For induction to work in the arguments below, it is worthwhile to consider blowings up of  centers of the  form $$\bar J^{1/c}:= (x_1^{1/(w_1c)},\ldots,x_k^{1/(w_kc)})$$ for a positive integer $c$.
We also use the notation $J^\alpha:= (x_1^{a_1\alpha},\ldots,x_k^{a_k\alpha})$ throughout --- this being an  equality  of valuative $\QQ$-ideals.

\subsubsection{Basic properties}

The description in Section \ref{Sec:admissibility-blowup} of the monomial valuation of $J$ immediately provides the following lemmas:

\begin{lemma}\label{Lem:admissible-operations}  If $J$ is both $\cI_1$-admissible and $\cI_2$-admissible then $J$ is $\cI_1+\cI_2$-admissible. If $J$ is $\cI$-admissible then $J^k$ is $\cI^k$-admissible. More generally if $J^{c_j}$ is $\cI_j$-admissible then $J^{\sum c_j}$ is $\prod \cI_j$-admissible.\end{lemma}

Indeed if $v_J(f)\geq 1$ and $v_J(g)\geq 1$ then $v_J(f+g)\geq 1$ and $v_J(f^{c_1}\cdot g^{c_2})\geq c_1+c_2$, etc.

\begin{lemma}\label{Lem:admissible-derivative} If $J$ is $\cI$-admissible then  $J'=J^{\frac{a_1-1}{a_1}}$ is $\cD(\cI)$-admissible. If $a_1>1$ and $J^{\frac{a_1-1}{a_1}}$ is $\cI$-admissible then $J$ is $x_1\cI$-admissible.
\end{lemma}
\begin{proof}
For the first statement note that if $\sum_{i=1}^k \alpha_i/a_i \geq 1$ and $\alpha_j\geq 1$ then $$v_J\left(\frac{\partial(x_1^{\alpha_1} \cdots x_n^{\alpha_n})}{\partial x_j}\right) \quad = \quad \sum_{i=1}^k \alpha_i/a_i\  -\  1/a_j \quad \geq \quad 1\  -\  1/a_1,$$ so $$v_{J'}\left(\frac{\partial(x_1^{\alpha_1} \cdots x_n^{\alpha_n})}{\partial x_j}\right) \quad \geq 1, $$ as needed. The other statement is similar.
\end{proof}

\begin{lemma}\label{Lem:admissible-lift} For \ChJarek{$\cI_0\subset k\llbracket x_2,\ldots,x_n\rrbracket$ write $\tilde\cI_0 = \cI_0k\llbracket x_1,\ldots, x_n\rrbracket$.} Assume 
$(x_2^{a_2},\ldots, x_k^{a_k})$ is $\cI_0$-admissible. Then $(x_1^{a_1},\ldots, x_k^{a_k})$ is $\tilde\cI_0$-admissible.
\end{lemma}

Here for generators of  $\cI_0$ we have $\sum_{i=1}^k \alpha_i/a_i = \sum_{i=2}^k \alpha_i/a_i$.

\begin{lemma}\label{Lem:admissible-coefficient} $J$ is $\cI$-admissible if and only if  $J^{(a_1-1)!}$ is $C(\cI,a_1)$-admissible.
\end{lemma}

\ChDan{
\begin{proof} When $\cI$ has order $<a_1$ then $J$ is not admissible for $\cI$ and $J^{(a_1-1)!}$ is not admissible for $\cC(\cI,a_1) = (1)$. When $\cI$ has order $\geq a_1$ this combines Lemmas \ref{Lem:admissible-operations} and \ref {Lem:admissible-derivative} for the terms defining    $C(\cI,a_1)$.
\end{proof}

This statement is only relevant, and will only be used, when $\cI$ has order $a_1$. If $a_1<a:=\ord(\cI)$ then $J^{(a_1-1)!}$ is in general not $C(\cI,a)$-admissible. For instance $J = (x_1)$ is admissible for $\cI=(x_1x_2)$ but not for $C(\cI,2) = (x_1^2,x_1x_2,x_2^2)$.}

\ChDan{

\begin{lemma}\label{Lem:admissible-transverse} Assume $(x_1,x_2,\ldots,x_n)$ and $(x_1',x_2,\ldots,x_n)$ are both regular sequences of parameters, and suppose $(x_1^{a_1},x_2^{a_2},\ldots,x_k^{a_k}) \leq v(x_1^{a_1})$. Then $(x_1^{a_1},x_2^{a_2}\ldots,x_k^{a_k}) = ({x_1'}^{a_1},x_2^{a_2}\ldots,x_k^{a_k})$ as centers.
\end{lemma}
\begin{proof} We may rescale $a_i$ and assume they are all integers. The inequality $(x_1^{a_1},x_2^{a_2}\ldots,x_k^{a_k}) \leq v({x_1'}^{a_1})$ implies that ${x_1'}^{a_1}$ lies in the integral closure $(x_1^{a_1},x_2^{a_2}\ldots,x_k^{a_k})^{int}$, hence $$({x_1'}^{a_1},x_2^{a_2}\ldots,x_k^{a_k})^{int}\subset (x_1^{a_1},x_2^{a_2}\ldots,x_k^{a_k})^{int}.$$ Since these two ideals have the same Hilbert--Samuel functions they coincide.
\end{proof}
}

\ChDan{
\subsection{Our chosen center is uniquely admissible}
\begin{theorem}\label{Th:center-admissible} 
\begin{enumerate} 
\item If $(a_1,\ldots ,a_k) = \inv_p(\cI)$, with corresponding parameters $x_1,\ldots,x_k$, and $J = (x_1^{a_1},\ldots,x_k^{a_k})$ a corresponding center, then $J$ is $\cI$-admissible.
\item  $$\inv_p(\cI) = \max_{ ({x_1'}^{b_1},\ldots,{x_k'}^{b_k}) \leq v(\cI)} (b_1,\ldots,b_k),$$
in other words it is the maximal invariant of a center admissible for $\cI$.
\item $J$ is the unique admissible center with invariant $\inv_p(\cI)$.
\end{enumerate}
\end{theorem}

\begin{proof} 
We first prove (1).  Applying Lemma \ref{Lem:admissible-coefficient}, we
replace $\cI$ by $\cC = C(\cI,a_1)$, rescale the invariant up to $a_1!$ and work on formal completion. We may therefore write
$$\cC = (x_1^{a_1!}) + (x_1^{a_1!-1} \tilde\cC_{1})  + \cdots + (x_1 \tilde\cC_{a_1!-1}) + \tilde\cC_{a_1!}$$
as in Proposition \ref{Lem:coefficient-description}.

The inductive hypothesis implies 
$J^{(a_1-1)!}$ is $\bar\cC_{a_1!}$-admissible. By Lemma \ref{Lem:admissible-lift} $J^{(a_1-1)!}$ is $\tilde\cC_{a_1!}$-admissible.
By Lemma \ref{Lem:admissible-derivative} $J^{(a_1-1)!}$ is  $(x_1^{a_1!-j} \tilde\cC_{j})$-admissible, so by Lemma \ref{Lem:admissible-operations}  $J^{(a_1-1)!}$ is $\cC$-admissible, as needed.

We prove (2) and (3) simultaneously. Let $J = (x_1^{a_1}, \dots,x_k^{a_k})$ be a center associated with the invariant at $p$.

Assume $(b_1,\ldots,b_m)\geq (a_1, \ldots,a_k)$.  If $J' = ({x_1'}^{b_1},\ldots,{x_k'}^{b_k})$ is admissible for $\cI$ then $b_1\leq a_1$.  Since our chosen center $J$ has $b_1=a_1$ this maximum is achieved. Let  $\ell = \max \{i:b_i=a_1\} \geq 1$. Evaluating  $J' < v(\cI) \leq v(x^{a_1})$ at the divisorial valuation of $x_1=0$ we have that $x_1 \in (x_1',\ldots,x_\ell') + \fm_p^2$, and after reordering we get that $(x_1,x_2',\ldots,x_n')$ is a regular system of parameters. By Lemma \ref{Lem:admissible-transverse} we may write $J' = (x_1^{a_1}, {x_2'}^{b_2},\ldots,{x_k'}^{b_k})$. Working on formal completions we may replace $x_i'$ by a suitable $x_i' + \alpha x_1$ so we may assume $x_i' \in k\llbracket x_2,\ldots,x_n\rrbracket$.

By Lemma \ref{Lem:admissible-coefficient} we may replace $\cI$ by $C(\cI, a_1)$ and replace  $\inv_p(\cI)$ by $(a_1-1)! \ (a_1,\ldots,a_k)$. Working again on formal completions we write
$$\cC = (x_1^{a_1!}) + (x_1^{a_1!-1} \tilde\cC_{1})  + \cdots + (x_1 \tilde\cC_{a_1!-1}) + \tilde\cC_{a_1!}.$$
By induction $(a_1-1)! \ (a_2,\ldots,a_k)$ is the maximal invariant for $\bar\cC_{a_1!}$, with unique center $(x_2^{a_2},\ldots,x_k^{a_k})$.  By functoriality, the invariant is maximal for $\tilde\cC_{a_1!}$. But $J' =(x_1^{a_1}, {x_2'}^{b_2},\ldots,{x_k'}^{b_k}) < v(\tilde\cC_{a_1!})$ is equivalent to $( {x_2'}^{b_2},\ldots,{x_k'}^{b_k}) < v(\tilde\cC_{a_1!})$.   It follows that $(a_1-1)!\ (a_1,\ldots,a_k)$ is the maximal invariant of a center admissible for  $C(\cI, a_1)$, with unique center $J$.

\end{proof}

}
    
\begin{remark}\label{Rem:flexible} 
\begin{enumerate} 
\item Stated in terms of the monomial valuation $v_J$ associated to $J$, the theorem says it is the unique monomial valuation with lexicograohically \emph{minimal} weights $(w_1,\ldots,w_n)$ satisfying $v(\cI)= 1$.
\item As an example for the added flexibility provided by admissibility, the center $(x_1^6,x_2^6)$ is $(x_1^3x_2^3)$-admissible because this is the corresponding invariant, but  also $(x_1^5,x_2^{15/2})$ is admissible. This second center becomes important when one considers instead the ideal $(x_1^5+x_1^3x_2^3)$, or even $(x_1^5+x_1^3x_2^3 +x_2^8)$, whose invariant is $(5, 15/2)$, as described in Section \ref{Sec:examples} below.
\end{enumerate}
\end{remark}

\ChDan{
%

\begin{corollary}\label{Cor:invariant-homogeneous}
We have $\inv_p(\cI^k) = k\cdot\inv_p(\cI)$ and $\inv_p(C(\cI, a_1)) = (a_1-1)! \ \inv_p(\cI)$ when $a_1 = \ord_p(\cI)$.
\end{corollary}
Indeed $J^k$ is admissible for $\cI^k$ if and only if $J$ is admissible for $\cI$, and Lemma \ref{Lem:admissible-coefficient} provides the analogous statement for the coefficient ideal.

}

\section{Principalization}

\subsection{The principalization theorem}
\begin{theorem}[Principalization]\label{Th:principalization} There is a functor $F_{pr}$ associating to a nowhere zero  $ \cI\subsetneq \cO_Y$ an admissible center $J$ with reduced center $\bar J$, with blowing up $Y'\to Y$ and \ChDan{weak} transform $ \cI'\subset \cO_Y'$, such that $\maxinv(\cI') <\maxinv(\cI)$.
In particular there is an integer $n$ so that the iterated application $(\cI_n \subset \cO_{Y_n}):= F_{pr}^{\,\circ n}(\cI\subset \cO_Y)  $ of $F_{pr}$ has $\cI_n = (1)$. Functoriality here is with respect to smooth surjective morphisms.

The stabilized functor $F_{pr}^{\,\circ \infty}(\cI\subset \cO_Y)$ is functorial for all smooth morphisms, whether or not surjective.
\end{theorem}

\ChDan{The notion of weak transform is introduced in Section \ref{Sec:admissibility-blowup}.}

\subsection{The invariant drops}
With admissibility of the center we can now analyze  the behavior of the invariant under the corresponding blowing up:

\begin{theorem}\label{Th:invariant-drops} Assume $\cI_p\neq (1)$, and let $(a_1,\ldots ,a_k) = \inv_p(\cI)$, with corresponding parameters $x_1,\ldots,x_k$, and $J = (x_1^{a_1},\ldots,x_k^{a_k})$. For $c\in \NN_{>0}$ write $Y'_c\to Y$ for the blowing up of the rescaled center   $\bar J^{1/c}:= \big(x_1^{1/(w_1c)},\ldots,x_k^{1/(w_kc)}\big)$, with corresponding factorization $\cI\cO_{Y'_c} = E^{a_1w_1c}\cI'$.  Then for every point $p'$ over $p$ we have $\inv_{p'}(\cI') < \inv_p(\cI)$.
\end{theorem}


\begin{proof}
If $k=0$ the ideal is $(0)$ and there is nothing to prove. When $k=1$ the ideal is $(x_1^{a_1})$, which becomes exceptional with weak transform $\cI'=(1)$. We now assume $k>1$.

Again using  Proposition \ref{Lem:coefficient-description}, we choose formal coordinates, work with $\tilde\cC:= C(\cI,a_1)$,
and write $$\tilde\cC = (x_1^{a_1!}) + (x_1^{a_1!-1} \tilde \cC_{1})  + \cdots + (x_1 \tilde \cC_1) + \tilde \cC_{a_1!}.$$ Writing $\tilde\cC \cO_{Y'_c} = E^{a_1!w_1c} \tilde\cC'$, we will first show that $\inv_{p'} (\tilde\cC') < (a_1-1)!\cdot(a_1,a_2,\ldots,a_k)$ for all points $p'$ over $p$.

Write $H = \{x_1=0\}$, and $H'\to H$ the blowing up of the reduced center $\bar J_H$ associated to $J_H := (x_2^{a_2}, \ldots, x_k^{a_k})$. By Lemma \ref{Lem:root} the proper transform $\tilde H' \to H$ of $H$ via the blowing up of $\bar J$
 is the blowing up of $\bar J_H^{1/(cc')}$, allowing for induction.

We now inspect the behavior on different charts.  On the $x_1$-chart we have $x_1 = u^{w_1c}$ so the first term becomes  $(x_1^{a_1!}) = E^{a_1!w_1c}\cdot(1)$ and $\inv_{p'} \tilde\cC' = \inv(1) = 0$.\footnote{This reflects the fact that before passing to the coefficient ideal $\ord(\cI')<a_1$ on this chart --- it need not become a unit ideal in general!} This implies that on all other charts it suffices to consider $p'\in \tilde H'\cap E$, as all other points belong to the $x_1$-chart. By the inductive assumption, for such points
we have  $$\inv_{p'}((\bar\cC_{a_1!})') <(a_1-1)!\cdot(a_2,\ldots,a_k).$$




Note that the term $(x_1^{a_1!})$ in $\tilde\cC$ is transformed, via $x_1 = u^{w_1c}x_1'$ to  the form $E^{a_1!w_1c}({x'_1}^{a_1!})$. It follows that $\ord_{p'} (\tilde\cC') \leq a_1!$, and  if $\ord_{p'} (\tilde\cC') < a_1!$ then a fortiori $\inv_{p'} (\tilde\cC') < \inv_{p} (\tilde\cC)$.

If on the other hand $\ord_{p'} (\tilde\cC') = a_1!$ then the variable $x_1'$ is a maximal contact element.
Using the inductive assumption we compute $$\inv_{p'}(({x'}_1^{a_1!}) + (\tilde\cC_{a_1!})') = \left(a_1!,\inv_{p'} ((\bar\cC_{a_1!})')\right) < (a_1!,\inv_{p'}(\bar\cC_{a_1!})) = (a_1-1)!(a_1,\ldots,a_k).$$ Since $\tilde\cC'$ includes this ideal, we obtain again $\inv_{p'} (\tilde\cC') < \inv_{p} (\tilde\cC)$, as claimed.

\ChDan{We deduce that $\inv_{p'}(\cI') < \inv_p(\cI)$ as well. As in \cite[Lemma 3.3]{Bierstone-Milman-funct}, \cite{ATW-principalization,ATW-relative}, we have the inclusions $\cI'^{(a_1-1)!} \subset \tilde\cC' \subset C(\cI',a_1),$\footnote{These are the ``easy" inclusions --- which hold even in the logarithmic situation.} 
hence $\ord_{p'}(\cI') \leq a_1$. We may again assume $x_1'$ is a maximal contact element and $\ord_{p'}(\cI') = a_1$. By Theorem \ref{Th:center-admissible}(2)  $$\inv_{p'}(\cI'^{(a_1-1)!}) \geq \inv_{p'}(\tilde\cC') \geq  \inv_{p'}(C(\cI',a_1)).$$
By Corollary \ref{Cor:invariant-homogeneous} we have $\inv_{p'}(\cI'^{(a_1-1)!}) = \inv_{p'}(C(\cI',a_1))$ giving equalities throughout, hence $$\inv_{p'}(\cI') = \frac{1}{(a_1-1)!} \inv_{p'}(\tilde\cC') < \frac{1}{(a_1-1)!} \inv_p(\tilde\cC) = \inv_p(\cI),$$
as needed.}

 \end{proof}

\begin{proof}[Proof of Theorem \ref{Th:principalization}]
The first paragraph of Theorem \ref{Th:principalization} follows from Theorems  \ref{Th:invariant-drops} and \ref{Th:center-admissible}(3), so we address the second paragraph with the following standard observation. Let $Y_n \to \cdots \to Y$ be the principalization of $\cI\subset \cO_Y$ and $\tilde Y \to Y$ a smooth morphism with $\tilde\cI = \cI\cO_{\tilde Y}$. Then the principalization of $\tilde \cI\subset \cO_{\tilde Y}$ is obtained from  $Y_n \to \cdots \to Y$  by removing empty blowings up.
\end{proof} 

\begin{proof}[Proof of Theorem \ref{maintheorem}] Apply Theorem \ref{Th:principalization}  to $\cI_X$, \ChDan{with the following  standard reduction. At each step one replaces the weak transform $\cI'$ by the proper transform $\cI_{X'} \supset \cI'$, noting that $\maxinv(\cI_{X'}) \leq \maxinv (\cI')$. One stops}  at the point where $\maxinv(\cI_{X_n}) = (1,\ldots,1)$, the sequence of length $c$: at this point the center $J_n$, whose support is contained in $X_n$, is everywhere  of the form $(x_1,\ldots,x_c)$, in particular smooth. Since $X_n$ is of pure codimension $c$, and since $\inv_p(\cI_{X_n}) = (1,\ldots,1)$ at a smooth point of $X_{n}$, we have that the support of $J_n$ contains a dense open in $X_n$, hence they coincide, and $X_n$ is smooth.
\end{proof}

 \section{Two examples}\label{Sec:examples}
 Consider the plane curve $$X = V(x^5 + x^3 y^3 + y^k)$$ with $k\geq 5$. Its resolution depends on whether or not $k\geq 8$.
\subsection{The case $k\geq 8$} This curve is singular at the origin $p$. We have $a_1 = \ord_p(\cI_X) = 5$. Since $\cD^{\leq 4} \cI = (x,y^2)$ we may take $x_1 = x$ and $H = V(x)$.
		A direct computation provides the coefficient ideal $$C(\cI_X,5)|_{H} = (\cD^{\leq 3}(\cI_X)|_H)^{120/2} = (y^{180}),$$ with $b_2 = 180$ and $a_2 = 180/(4!) = 15/2$. Rescaling, we need to take the weighted blowup of $\bar J = (x^{1/3},y^{1/2})$.
		\begin{itemize}
			\item In the $x$-chart we have $x= u^3, y= u^2y'$, giving $$Y'_x = [\Spec k[u,y']/\bmu_3],$$ the action given by $(u,y') \mapsto (\zeta_3u, \zeta_3y')$.
				The equation of $X$ becomes $$u^{15}(1+{y'}^3 + u^{2k-15}{y'}^k),$$ with proper transform $X'_x = V(1+{y'}^3 + u^{2k-15}{y'}^k)$ smooth.
			\item In the $y$-chart we have $y = v^2, x=v^3x'$, giving 	$$Y'_y = [\Spec k[x',v]/\bmu_2],$$ the action given by $(x',v) \mapsto (-x', -v)$. The equation of $X$ becomes $v^{15}({x'}^5 + {x'}^3 + v^{2k-15})$, with proper transform $X'_y = V({x'}^5 + {x'}^3 + v^{2k-15})$.
			
			Note that $X'_y$ is smooth when $k=8$. Otherwise it is singular at the origin with invariant $(3,2k-15)$, which is lexicographically strictly smaller than $(5,15/2)$; A single weighted blowing up resolves the singularity.
		\end{itemize}	
\subsection{The case $k\leq 7 $} 		
		 Consider now the same equation with $k=7$ (the cases $k=5,6$ being similar). We still take $a_1=5, x_1 = x$ and $H = V(x)$. This time $$C(\cI_X)|_{H} = ((\cI_X)|_H)^{120/5} = (y^{168}),$$ with $b_2 = 7\cdot (4!)$ and $a_2 =  7$. We take the weighted blowup of $J = (x^{1/7},y^{1/5})$.
		\begin{itemize}
			\item In the $x$-chart we have $x= u^7, y= u^5y'$, giving
			$$Y'_x = [\Spec k[u,y']/\bmu_7],$$ the action given by $(u,y') \mapsto (\zeta_7u, \zeta_7^{-5}y')$.
				The equation of $X$ becomes $$u^{35}(1+u{y'}^3 + {y'}^7),$$ with proper transform $X'_x = V(1+u{y'}^3 + {y'}^7)$ smooth.
			\item In the $y$-chart we have $y = v^5, x=v^7x'$, giving 	$$Y'_y = [\Spec k[x',v]/\bmu_5],$$ the action given by $(x',v) \mapsto (\zeta_5^{-7}x', \zeta_5v)$. The equation of $X$ becomes $v^{35}({x'}^5 + v{x'}^3 + 1)$, with smooth proper transform $X'_y = V({x'}^5 + v{x'}^3 + 1)$.
		\end{itemize}

\section{Further comments}\label{Sec:comments}


\subsection{Non-embedded resolution} Given two embeddings $X \subset Y_1$ and $X \subset Y_2$ such that $\dim_p(Y_1) = \dim_p(Y_2)$ for all $p\in X$, the two embeddings are \'etale locally equivalent. By functoriality the embedded resolutions of $X \subset Y_1$ and $X \subset Y_2$ are \'etale locally isomorphic, hence the resolutions $X'_1 \to X$ and $X'_2 \to X$ coincide.

Our resolutions also satisfy the re-embedding principle \cite[proposition 2.12.3]{ATW-principalization}: given an embedding $Y \subset Y_1 := Y\times \Spec k[x_0]$ and $\inv_p(\cI_{X \subset Y}) = (a_1,\ldots,a_k)$ with parameters $(x_1,\ldots, x_k)$ we have $\inv_p(\cI_{X \subset Y_1}) = (1,a_1,\ldots,a_k)$ with parameters $(x_0,x_1,\ldots, x_k)$. The proper transform $X'_1$ of $X$ in $Y_1'$ is disjoint from the $x_0$-chart, and on every other chart we have $Y_1' = Y' \times \Spec k[x_0]$ so that $X'_1 = X'$ and induction applies.

Since every pure-dimensional stack can be \'etale locally embedded in pure codimension, we deduce:
\begin{theorem}[Non-embedded resolution]\label{Th:non-embedded}  There is a functor $F_{ner}$ associating to a pure-dimensional reduced stack $X$ of finite type over a characteristic-0 field $k$ a proper, generically representable and  birational morphism $F_{ner}(X) \to X$ with $F_{ner}(X)$ regular. This is functorial for smooth morphisms: if $X_1 \to X$ is smooth then $F_{ner}(X_1) = F_{ner}(X)\times_X X_1$.
\end{theorem}
\ChDan{Of course one can deduce resolution of $X$ which is not pure dimensional, though care is required for functoriality.}

Carefully using Bergh's destackification theorem we also obtain:
\begin{theorem}[Coarse resolution]\label{Th:coarse-resolution}  There is a functor $F_{crs}$ associating to a pure-dimensional reduced stack $X$ of finite type over a characteristic-0 field $k$ a \emph{projective}  birational morphism $F_{crs}(X) \to X$ with $F_{crs}(X)$ regular. This is functorial for smooth morphisms: if $X_1 \to X$ is smooth  then $F_{crs}(X_1) = F_{crs}(X)\times_X X_1$.
\end{theorem}
\begin{proof} We apply \cite[Theorem 7.1]{Bergh-Rydh}, using $F_{ner}(X)\to X \to \Spec k$ for $X \to T \to S$ in that theorem. This provides a projective morphism $F_{ner}(X)' \to F_{ner}(X)$, functorial for  smooth morphisms $X_1 \to X$, such that the relative coarse moduli space $F_{ner}(X)' \to \underline{F_{ner}(X)'} \to X$ is projective over $X$, and such that $F_{ner}(X)' $ and $\underline{F_{ner}(X)'}$ are regular. We may take $F_{crs}(X)  = \underline{F_{ner}(X)'}$.
\end{proof}
\subsection{Note on stabilizers} Even though Bergh's destackification is known for tame stacks,  one might wonder about the stabilizers occurring in our resolution. We note, however, that the stabilizers of a weighted blowing up locally embed in $I_Y \times \GG_m$, where $I_Y$ denotes the inertia stack of $Y$. We therefore have that the stabilizers of $Y_n$ locally embed in $I_Y \times \GG_m^n$. In particular, if $Y$ is a scheme then $Y_n$ has abelian inertia, and its coarse moduli space has abelian quotient singularities.
\subsection{Note on exceptional loci}\label{Sec:not-dnc} We show by way of an example that the exceptional loci produced in our algorithm do not necessarily have normal crossings with centers.

Consider $\cI=(x^2yz+yz^4) \subset \CC[x,y,z]$. Then $\maxinv(\cI) = (4,4,4)$ is attained at the origin with center $(x^4,y^4,z^4)$ and reduced center $(x,y,z)$. In the $z$-chart one obtains the ideal $(y_3(x_3^2 + z))$. The new invariant is $(2,2)$ with reduced center $(y_3,x_3^2+z)$, which is tangent to the exceptional $z=0$.

The methods of \cite{ATW-principalization} suggest using the logarithmic derivative in $z$, resulting in the invariant $(3,3,\infty)$ with center $(y_3^3,x_3^3,z^{3/2})$ and reduced Kummer center $(y_3,x_3,z^{1/2})$. This reduces logarithmic invariants respecting logarithmic, hence exceptional, divisors.

\bibliographystyle{amsalpha}
\bibliography{principalization}

\end{document}